\documentclass{fancypaper-cm}
\usepackage[extdef=true]{delimset}
\DeclareMathDelimiterSet{\scal}[2]{
  \selectdelim[l]< {#1} \mathpunct{}\selectdelim[p]| {#2} \selectdelim[r]>
}

\newcommand{\menge}[2]{
  \bigl\{{#1}\mid{#2}\bigr\}
}
\DeclareMathDelimiterSet{\Menge}[2]{
  \selectdelim[l]\{{#1}\selectdelim[m]|{#2}\selectdelim[r]\}
}

\DeclareMathDelimiterSet{\floor}[1]{
  \selectdelim[l]\lfloor {#1} \selectdelim[r]\rfloor
}

\newcommand{\NN}{\mathbb{N}}
\newcommand{\RR}{\mathbb{R}}

\newcommand{\HH}{\mathcal{H}}

\newcommand\restr[2]{{
  \left.\kern-\nulldelimiterspace 
    #1 
  \right|_{#2} 
  }
}

\newcommand{\pinf}{{+}\infty}

\newcommand{\emp}{\varnothing}

\renewcommand{\leq}{\leqslant}
\renewcommand{\geq}{\geqslant}

\newcommand{\zeroun}{\intv[o]{0}{1}}
\newcommand{\rzeroun}{\intv[l]{0}{1}}
\newcommand{\lzeroun}{\intv[r]{0}{1}}

\newcommand{\RP}{\intv[r]{0}{\pinf}}
\newcommand{\RPP}{\intv[o]{0}{\pinf}}

\DeclareMathOperator{\ran}{ran}

\DeclareMathOperator{\zer}{zer}

\newcommand{\Id}{\mathrm{Id}}

\newcommand{\weakly}{\rightharpoonup}

\DeclareMathOperator{\ccone}{\overline{cone}}

\DeclareMathOperator{\proj}{proj}

\newcommand{\minimize}[2]{
  \underset{\substack{{#1}}}{\operatorname{minimize}}\;\;#2
}

\crefname{cexample}{Counterexample}{Counterexamples}
\newtheorem{cexample}[theorem]{Counterexample}

\author{Minh N. B\`{u}i}
\affil{%
Chongqing Normal University
\affilcr
School of Mathematical Sciences
\affilcr
Shapingba District, Chongqing 401331, P.~R.~China
\affilcr
\email{mnb@cqnu.edu.cn}
}

\author{Honglin Luo}
\affil{%
Chongqing Normal University
\affilcr
School of Mathematical Sciences
\affilcr
Shapingba District, Chongqing 401331, P.~R.~China
\affilcr
\email{luohonglin@cqnu.edu.cn}
}

\author{Xianfu Wang}
\affil{%
University of British Columbia
\affilcr
Irving K. Barber Faculty of Science,
Department of Mathematics
\affilcr
Kelowna, British Columbia V1V~1V7, Canada
\affilcr
\email{shawn.wang@ubc.ca}
}

\title{The Douglas--Rachford algorithm for\\ normal cones
converges only weakly\thanks{Corresponding author:
M.~N.~B\`{u}i (\email{mnb@cqnu.edu.cn}).
}}

\begin{document}

\date{~}

\thispagestyle{plain.scrheadings}

\maketitle

\begin{abstract}
We show that the Douglas--Rachford algorithm applied to two normal
cone operators associated with closed convex sets can fail to
converge strongly, even when one set is a closed vector subspace,
the other is a closed convex cone, and their intersection is a
singleton.
\end{abstract}


The Douglas--Rachford algorithm \cite{Douglas56,Lions79}
(see \cite[Section~6.1]{Combettes24} for a historical account)
is a fundamental method for finding a zero of the sum of maximally
monotone operators acting on a real Hilbert space, a problem
arising in various areas of applied mathematics, including
optimization,
machine learning,
signal processing,
statistics,
optimal control,
and optimal transportation
\cite{Combettes24}.
To facilitate our discussion, let us state its convergence
property (see, e.g., \cite[Theorem~26.11]{Livre1} for a proof);
we refer the reader to \cite{Livre1} for background
on monotone operator theory, as well as for notation not explicitly
defined herein.
Recall that the resolvent of a maximally monotone set-valued
operator $M$ acting on a real Hilbert space is
\begin{equation}
J_M=(\Id+M)^{-1}.
\end{equation}

\begin{theorem}
\label{t:1}
Let $\HH$ be a real Hilbert space, and let
$A$ and $B$ be maximally monotone operators from
$\HH$ to $2^{\HH}$ such that
$\zer(A+B)
=\menge{x\in\HH}{0\in Ax+Bx}
\neq\emp$.
Let $y_0\in\HH$ and iterate
\begin{equation}
\begin{array}{l}
\text{for $n=0,1,\ldots$}\\
\left\lfloor
\begin{array}{l}
x_n=J_By_n\\
y_{n+1}=y_n+J_A(2x_n-y_n)-x_n.
\end{array}
\right.
\end{array}
\end{equation}
Then there exists $(x,y)\in\HH\times\HH$ such that
\begin{enumerate}
\item
\label{t:1i}
$x=J_By$, $x=J_A(2x-y)$, and $x\in\zer(A+B)$.
\item
\label{t:1ii}
$x_n\weakly x$ and $y_n\weakly y$.
\end{enumerate}
\end{theorem}

In the general setting of \cref{t:1}, the work \cite{Sicon20}
demonstrates that the weak convergence in item \cref{t:1ii}
cannot be strengthened to strong convergence.
More precisely, using the closed vector subspace $V$ and the closed
convex cone $K$ constructed by Hundal in his study of
the method of alternating projections \cite{Hundal04},
\cite{Sicon20} establishes a counterexample with
\begin{equation}
\label{e:ABsico}
A=N_V
\quad\text{and}\quad
B=(\proj_V\circ\proj_K\circ\proj_V)^{-1}-\Id;
\end{equation}
see \cref{cex:1} below for the definitions of $V$ and $K$.
Here, given a nonempty closed convex subset $C$ of a real Hilbert
space $\HH$,
the normal cone operator associated with $C$ is the maximally
monotone operator
\begin{equation}
N_C\colon\HH\to 2^{\HH}\colon x\mapsto
\begin{cases}
\menge{u\in\HH}{(\forall y\in C)\,\,\scal{y-x}{u}\leq 0},
&\text{if}\,\,x\in C;\\
\emp,&\text{otherwise},
\end{cases}
\end{equation}
and the projector $\proj_C=(\Id+N_C)^{-1}$ maps every point
$x\in\HH$ to the unique point $\proj_Cx\in C$ such that
\begin{equation}
\norm{x-\proj_Cx}=\min_{y\in C}\norm{x-y}.
\end{equation}
Since the operator $B$ in \cref{e:ABsico} is not a normal cone
operator, it remains an open question whether the weak convergence
of the sequences $(x_n)_{n\in\NN}$ and $(y_n)_{n\in\NN}$ in
\cref{t:1}\cref{t:1ii} can be strengthened to strong convergence
when both $A$ and $B$ are normal cone operators associated with
closed convex subsets of $\HH$, a setting of interest in
applications such as image recovery \cite{Combettes96}
and signal reconstruction \cite{Thao21}.
In this setting, \cite[Corollary~4.5]{Bauschke14}
establishes strong convergence of these sequences when both sets
are affine subspaces. We complement \cite{Sicon20} by
employing the sets $V$ and $K$ to show that strong convergence can
fail even when one set is a closed vector subspace, the other is a
closed convex cone, and their intersection is a singleton.

\begin{cexample}
\label{cex:1}
Let $\HH$ be a separable infinite-dimensional real Hilbert space,
let $(e_k)_{k\in\NN}$ be an orthonormal basis of $\HH$,
and set
\begin{equation}
\begin{cases}
V=\set{e_0}^{\perp}\\
\displaystyle
K=\ccone\menge{
h(\xi)e_0
+\cos\brk!{\phi(\xi)}e_{\floor{\xi}+1}
+\sin\brk!{\phi(\xi)}e_{\floor{\xi}+2}
}{\xi\in\RP},
\end{cases}
\end{equation}
where $\floor{\xi}$ denotes the integer part of $\xi\in\RP$ and
\begin{equation}
h\colon\RP\to\RR\colon\xi\mapsto\exp\brk{-100\xi^3}
\quad\text{and}\quad
\phi\colon\RP\to\RR\colon\xi\mapsto
\frac{\pi\brk{\xi-\floor{\xi}}}{2}.
\end{equation}
Set $y_0=e_2$ and iterate
\begin{equation}
\label{e:1}
\begin{array}{l}
\text{for $n=0,1,\ldots$}\\
\left\lfloor
\begin{array}{l}
x_n=\proj_Vy_n\\
y_{n+1}=y_n+\proj_K(2x_n-y_n)-x_n.
\end{array}
\right.
\end{array}
\end{equation}
Then $V\cap K=\set{0}$ and the following hold:
\begin{enumerate}
\item
$(x_n)_{n\in\NN}$ converges weakly---but not strongly---to $0$.
\item
$(y_n)_{n\in\NN}$ converges weakly---but not strongly---to a point
in $\RPP\cdot e_0$.
\end{enumerate}
\end{cexample}

In the remainder of the paper, we work in the setting of
\cref{cex:1} and employ the notation
\begin{equation}
(\forall\xi\in\RP)\quad
u(\xi)
=\cos\brk!{\phi(\xi)}e_{\floor{\xi}+1}
+\sin\brk!{\phi(\xi)}e_{\floor{\xi}+2}
\quad\text{and}\quad
p(\xi)=h(\xi)e_0+u(\xi).
\end{equation}
Furthermore, let
\begin{equation}
\label{e:drop}
T=\Id+\proj_K\circ(2\proj_V-\Id)-\proj_V
\end{equation}
be the associated \emph{Douglas--Rachford operator}.

To prove \cref{cex:1}, let us first examine the iterations of
\cref{e:1} ``na\"{i}vely.''
Since $x_0=y_0=e_2=u(1)$, we learn from
\cite[Theorem~14]{Hundal04} that there exist
$\eta\in\intv[o]{1}{2}$ and $\alpha\in\zeroun$ such that
$\proj_K(2x_0-y_0)
=\proj_Ke_2
=\alpha p(\eta)$, which leads to $y_1=\alpha p(\eta)$
and $x_1=\proj_Vy_1=\alpha u(\eta)$. In turn, determining
$y_2$ requires calculating
\begin{equation}
\proj_K(2x_1-y_1)
=\proj_K\brk!{2\alpha u(\eta)-\alpha p(\eta)}
=\alpha\proj_K\brk{u(\eta)- h(\eta)e_0}.
\end{equation}
This motivates us to investigate the projection
\begin{equation}
\proj_K\brk{u(\xi)-\beta e_0}
\end{equation}
for a given $\xi\in\intv[r]{1}{\pinf}$ and a given
$\beta\in\intv{0}{1}$, which will be done in \cref{p:1}.
To this end, we follow the approach of
\cite[Theorem~14]{Hundal04}, which treats the case
$\proj_Ku(\xi)$.

We begin with four elementary lemmas that will be useful in
analyzing this projection. We leave the proofs of
\cref{l:0,l:3,l:1} to the reader.

\begin{lemma}
\label{l:0}
The following hold:
\begin{enumerate}
\item
\label{l:0iii}
$(\forall\xi\in\intv{0}{1})$
$\sin(\pi\xi/2)\geq\xi$.
\item
\label{l:0iv}
$(\forall\xi\in\intv{0}{\pi/2})$
$\cos\xi\geq 1-\xi^2/2$.
\item
\label{l:0v}
$(\forall\xi\in\intv[o]{0}{\pi/2})$
$\cot\xi<1/\xi$.
\end{enumerate}
\end{lemma}

\begin{lemma}
\label{l:3}
Let $a\in\RR$ and $b\in\RR$ be such that $a<b$,
let $f\colon\intv{a}{b}\to\RP$ be continuous and decreasing,
and let $(\xi_i)_{0\leq i\leq m+1}$ be a strictly increasing sequence
in $\intv{a}{b}$. Then
\begin{equation}
\sum_{i=0}^mf(\xi_{i+1})(\xi_{i+1}-\xi_i)
\leq\int_a^bf(t)dt.
\end{equation}
\end{lemma}

\begin{lemma}
\label{l:2}
Let $f\colon\RPP\to\RPP$ be a differentiable function that is
convex on $\intv[r]{1}{\pinf}$, and
let $\xi\in\intv[r]{1}{\pinf}$ and $\eta\in\intv[r]{1}{\pinf}$ be
such that $\abs{\xi-\eta}\leq 1$.
Then
\begin{equation}
\frac{\pi}{2}f(\xi)
-\frac{\pi}{2}f(\eta)\cos\frac{\pi(\xi-\eta)}{2}
-f'(\eta)\sin\frac{\pi(\xi-\eta)}{2}
\geq 0.
\end{equation}
\end{lemma}
\begin{proof}
Since the assertion is trivial when $\xi=\eta$,
we assume that $\xi\neq\eta$.
Set
\begin{equation}
\theta=\frac{\pi(\xi-\eta)}{2}
\quad\text{and}\quad
\tau=\frac{\sin\theta}{\theta}.
\end{equation}
Because $\theta\in\intv{-\pi/2}{\pi/2}\smallsetminus\set{0}$,
we deduce that $\tau\in\zeroun$ and that
$\cos\theta\leq\tau$. Therefore, since the convexity of $f$ on
$\intv[r]{1}{\pinf}$ gives
$f(\xi)-f(\eta)\geq f'(\eta)(\xi-\eta)$,
we get
\begin{align}
f(\xi)-f(\eta)\cos\theta-\frac{2}{\pi}f'(\eta)\sin\theta
&=f(\xi)-f(\eta)\cos\theta-\tau f'(\eta)(\xi-\eta)
\nonumber\\
&\geq f(\xi)-f(\eta)\cos\theta-\tau\brk!{f(\xi)-f(\eta)}
\nonumber\\
&=(1-\tau)f(\xi)+(\tau-\cos\theta)f(\eta)
\nonumber\\
&\geq 0,
\end{align}
as claimed.
\end{proof}

\begin{lemma}
\label{l:1}
The following hold:
\begin{enumerate}
\item
\label{l:1i}
$(\forall\xi\in\RP)$
$\phi(\xi)\in\intv[r]{0}{\pi/2}$.
\item
\label{l:1ii}
$h$ is decreasing and $\ran h=\rzeroun$.
\item
\label{l:1iii}
$(\forall\xi\in\RPP)$ $h'(\xi)<0$.
\item
\label{l:1iv}
$(\forall\xi\in\intv[r]{1}{\pinf})$
$h''(\xi)>0$.
\item
\label{l:1v}
$(\forall\xi\in\intv[r]{1}{\pinf})$
$\xi^4\exp(-100\xi^3)\leq\exp(-100)$.
\end{enumerate}
\end{lemma}

Let $\xi\in\intv[r]{1}{\pinf}$ and $\beta\in\intv{0}{1}$,
and let us sketch our approach to analyzing
\begin{equation}
\proj_K\brk{u(\xi)-\beta e_0}.
\end{equation}
A closer look at the approach of
\cite[Theorem~14]{Hundal04}, which treats the case
$\beta=0$, suggests that, under a suitable assumption on $\xi$,
this projection is of the form
$\overline{\alpha}p(\overline{\eta})$, where
$(\overline{\alpha},\overline{\eta})$ solves
\begin{equation}
\minimize{\substack{
\eta\in\intv[o]{\floor{\xi}}{\floor{\xi}+1}\\
\alpha\in\intv[o]{0}{1}
}}{\norm!{\alpha p(\eta)-(u(\xi)-\beta e_0)}^2}.
\end{equation}
The corresponding first-order optimality conditions lead to the
equations defining the candidate parameters $\eta$ and $\alpha$ in
the following result. We will also derive several
estimates on $\alpha$ and $\eta$ that will be useful later.

\begin{proposition}
\label{p:1}
Let $k\in\NN\smallsetminus\set{0}$, set $\delta=1/(300k^2)$,
let $\xi\in\intv{k}{k+1-2\delta}$, and let $\beta\in\intv{0}{1}$.
Then there exists a unique $\eta\in\intv[o]{\xi}{k+1-\delta}$
such that
\begin{equation}
\label{e:beeq}
\beta
={-}h(\eta)\cos\omega
-\frac{\pi}{2}\frac{1+h(\eta)^2}{h'(\eta)}\sin\omega,
\quad\text{where}\,\,\omega=\frac{\pi(\eta-\xi)}{2}.
\end{equation}
Now set
\begin{equation}
\label{e:zcfy}
\alpha=\cos\omega+
\frac{\pi}{2}\frac{h(\eta)}{h'(\eta)}\sin\omega.
\end{equation}
Then the following hold:
\begin{enumerate}
\item
\label{p:1i}
$\proj_K\brk{u(\xi)-\beta e_0}=\alpha p(\eta)$.
\item
\label{p:1ii}
$1-\alpha
<{-}5\pi h'(\eta)(\eta-\xi)/16
+\pi^2(\eta-\xi)/(1200\eta^2)$.
\item
\label{p:1iii}
$\alpha\in\intv[o]{1/2}{1}$.
\item
\label{p:1iv}
Set $\gamma=\beta/\alpha+h(\eta)$. Then
$\beta<\gamma$ and
$\gamma^2-\beta^2\leq\pi^2(\eta-\xi)/(50\eta^2)$.
\end{enumerate}
\end{proposition}
\begin{proof}
Set $\mu=k+1-\delta$, and define
$\theta\colon\RPP\to\RR\colon t\mapsto\pi(t-\xi)/2$
and
\begin{equation}
\varphi\colon\RPP\to\RR\colon t\mapsto
\frac{\pi}{2}\brk{1+h(t)^2}\sin\theta(t)
+\brk!{\beta+h(t)\cos\theta(t)}h'(t).
\end{equation}
Then
\begin{equation}
\label{e:2ds9}
\mu-\xi
\in\intv{\delta}{1-\delta},
\end{equation}
from which and \cref{l:0}\cref{l:0iii} we obtain
$\sin\theta(\mu)
\geq\sin(\pi\delta/2)
\geq\delta$.
In turn, since $\set{\beta,\cos\theta(\mu)}\subset\intv{0}{1}$,
it follows from \cref{l:1}\cref{l:1ii}\&\cref{l:1iii} that
\begin{align}
\varphi(\mu)
&=\frac{\pi}{2}\brk{1+h(\mu)^2}\sin\theta(\mu)
+\brk!{\beta+h(\mu)\cos\theta(\mu)}h'(\mu)
\nonumber\\
&\geq\frac{\pi\delta}{2}-
\brk!{\beta+h(\mu)\cos\theta(\mu)}\abs{h'(\mu)}
\nonumber\\
&\geq\frac{\pi\delta}{2}-2\abs{h'(\mu)},
\end{align}
while \cref{l:1}\cref{l:1v} entails that
\begin{align}
\label{e:2210}
\brk{\forall t\in\intv{\xi}{\mu}}\quad
2\abs{h'(t)}
&=600t^2\exp(-100t^3)
\nonumber\\
&\leq\frac{600\exp(-100)}{t^2}
\nonumber\\
&\leq\frac{600\exp(-100)}{k^2}
\nonumber\\
&<\frac{\pi}{600k^2}
\nonumber\\
&=\frac{\pi\delta}{2}.
\end{align}
Therefore $\varphi(\mu)>0$.
On the other hand, $\varphi(\xi)=\brk{\beta+h(\xi)}h'(\xi)<0$.
Therefore, since $\varphi$ is continuous on
$\intv{\xi}{\mu}$, the intermediate value
theorem guarantees that there exists
$\eta\in\intv[o]{\xi}{\mu}$ such that
$\varphi(\eta)=0$ or, equivalently, that \cref{e:beeq} holds.
To establish the uniqueness of $\eta$,
let us show that $\varphi$ is strictly increasing on
$\intv{\xi}{\mu}$.
For every $t\in\intv{\xi}{\mu}$, since
$0\leq\theta(t)\leq\theta(\mu)\leq\pi(1-\delta)/2$,
we deduce from \cref{l:0}\cref{l:0iii} that
$\cos\theta(t)
\geq\cos(\pi(1-\delta)/2)
=\sin(\pi\delta/2)
\geq\delta
>0$.
In turn, it results from \cref{l:1}\cref{l:1ii}--\cref{l:1iv} and
\cref{e:2210} that
\begin{align}
\brk{\forall t\in\intv{\xi}{\mu}}\quad
\varphi'(t)
&=\frac{\pi^2}{4}\brk{1+h(t)^2}\cos\theta(t)
+\frac{\pi}{2}\underbrace{h'(t)}_{<0}\underbrace{h(t)\sin\theta(t)}_{\in\intv[r]{0}{1}}
\nonumber\\
&\phantom{=}
+h''(t)\brk!{\beta+h(t)\cos\theta(t)}
+\brk!{h'(t)}^2\cos\theta(t)
\nonumber\\
&\geq\frac{\pi^2}{4}\cos\theta(t)-\frac{\pi}{2}\abs{h'(t)}
\nonumber\\
&\geq\frac{\pi}{2}\brk2{\frac{\pi\delta}{2}-\abs{h'(t)}}
\nonumber\\
&>0.
\end{align}
Therefore, $\varphi$ is strictly increasing on $\intv{\xi}{\mu}$.
To proceed with the remaining part of the proof, we observe that,
thanks to \cref{e:2ds9} and the fact that
$\eta\in\intv[o]{\xi}{\mu}$,
\begin{equation}
\label{e:hlxu}
\eta-\xi\in\intv[o]{0}{1-\delta}
\quad\text{and}\quad
\omega=\frac{\pi(\eta-\xi)}{2}
\in\intv[o]2{0}{\frac{\pi(1-\delta)}{2}}.
\end{equation}
Moreover, by \cref{e:beeq} and \cref{e:zcfy},
\begin{equation}
\label{e:vw30}
\frac{\cos\omega-\beta h(\eta)}{1+h(\eta)^2}
=\frac{\cos\omega}{1+h(\eta)^2}+
\frac{h(\eta)^2\cos\omega}{1+h(\eta)^2}
+\frac{\pi}{2}\frac{h(\eta)}{h'(\eta)}\sin\omega
=\cos\omega+\frac{\pi}{2}\frac{h(\eta)}{h'(\eta)}\sin\omega
=\alpha.
\end{equation}

\cref{p:1i}:
Since \cref{l:0}\cref{l:0v} yields
\begin{equation}
\tan\omega
=\cot\brk2{\frac{\pi}{2}-\omega}
<\frac{1}{\frac{\pi}{2}-\omega}
=\frac{2}{\pi\brk!{1-(\eta-\xi)}}
<\frac{2}{\pi\delta}
=\frac{600k^2}{\pi}
<\frac{600\eta^2}{\pi},
\end{equation}
it follows from \cref{e:zcfy} that
\begin{equation}
\alpha
=\cos\omega+\frac{\pi}{2}\frac{h(\eta)}{h'(\eta)}\sin\omega
=\brk2{1-\frac{\pi}{600\eta^2}\tan\omega}\cos\omega
>0.
\end{equation}
Hence, by construction of $K$, $\alpha p(\eta)\in K$.
Therefore, according to \cite[Proposition~6.28]{Livre1},
we must verify
\begin{equation}
\scal{p(\eta)}{u(\xi)-\beta e_0-\alpha p(\eta)}=0
\end{equation}
and that
\begin{equation}
\label{e:hekv}
(\forall x\in K)\quad
\scal{x}{u(\xi)-\beta e_0-\alpha p(\eta)}\leq 0.
\end{equation}
Since $\set{\xi,\eta}\subset\intv[r]{k}{k+1}$,
\begin{equation}
\label{e:hwh3}
\floor{\xi}=\floor{\eta}=k,
\end{equation}
which yields
\begin{equation}
\label{e:9ywu}
u(\xi)
=\cos\brk!{\phi(\xi)}e_{k+1}
+\sin\brk!{\phi(\xi)}e_{k+2}
\end{equation}
and
\begin{equation}
\label{e:9ywp}
p(\eta)
=h(\eta)e_0
+\cos\brk!{\phi(\eta)}e_{k+1}
+\sin\brk!{\phi(\eta)}e_{k+2}
=h(\eta)e_0+u(\eta).
\end{equation}
In turn, because $\set{e_0,e_{k+1},e_{k+2}}$ is an orthonormal set
in $\HH$, we derive from \cref{e:vw30} that
\begin{align}
\scal{p(\eta)}{u(\xi)-\beta e_0-\alpha p(\eta)}
&={-}\brk!{\beta+\alpha h(\eta)}h(\eta)
+\brk!{\cos\phi(\xi)-\alpha\cos\phi(\eta)}\cos\phi(\eta)
\nonumber\\
&\phantom{=}
+\brk!{\sin\phi(\xi)-\alpha\sin\phi(\eta)}\sin\phi(\eta)
\nonumber\\
&={-}\brk!{\beta+\alpha h(\eta)}h(\eta)
+\cos\brk!{\phi(\eta)-\phi(\xi)}
-\alpha
\nonumber\\
&=\cos\omega
-\brk!{\beta+\alpha h(\eta)}h(\eta)
-\alpha
\nonumber\\
&=\brk!{\cos\omega-\beta h(\eta)}-\alpha\brk{1+h(\eta)^2}
\nonumber\\
&=0.
\end{align}
Next, because $K=\ccone\menge{p(\tau)}{\tau\in\RP}$,
verifying \cref{e:hekv} is equivalent to establishing
\begin{equation}
\label{e:2xkt}
(\forall\tau\in\RP)\quad
\scal{p(\tau)}{u(\xi)-\beta e_0-\alpha p(\eta)}\leq 0.
\end{equation}
Toward this goal, define
\begin{equation}
v(\eta)
={-}\sin\brk!{\phi(\eta)}e_{k+1}+\cos\brk!{\phi(\eta)}e_{k+2}
\end{equation}
and
\begin{equation}
q(\eta)
=\frac{\pi}{2}e_0
-\frac{\pi}{2}h(\eta)u(\eta)
-h'(\eta)v(\eta).
\end{equation}
Since it results from \cref{e:hwh3} that
$\phi(\eta)-\omega=\phi(\xi)$,
we get from \cref{e:9ywu} that
\begin{align}
\cos(\omega)u(\eta)-\sin(\omega)v(\eta)
&=\brk!{\cos(\omega)\cos\brk{\phi(\eta)}
+\sin(\omega)\sin\brk{\phi(\eta)}}e_{k+1}
\nonumber\\
&\phantom{=}
+\brk!{\cos(\omega)\sin\brk!{\phi(\eta)}
-\sin(\omega)\cos\brk{\phi(\eta)}}e_{k+2}
\nonumber\\
&=\cos\brk!{\phi(\xi)}e_{k+1}+\sin\brk!{\phi(\xi)}e_{k+2}
\nonumber\\
&=u(\xi).
\end{align}
At the same time, we derive from
\cref{e:beeq,e:zcfy}
that
\begin{equation}
\beta+\alpha h(\eta)
={-}\frac{\pi}{2}\frac{\sin\omega}{h'(\eta)}
\quad\text{and}\quad
{-}\alpha+\cos\omega
={-}\frac{\pi}{2}\frac{h(\eta)}{h'(\eta)}\sin\omega.
\end{equation}
Hence, using \cref{e:9ywp}, we get
\begin{align}
u(\xi)-\beta e_0-\alpha p(\eta)
&={-}\brk!{\beta+\alpha h(\eta)}e_0
+(-\alpha+\cos\omega)u(\eta)
-\sin(\omega)v(\eta)
\nonumber\\
&=\frac{\pi}{2}\frac{\sin\omega}{h'(\eta)}e_0
-\frac{\pi}{2}\frac{h(\eta)}{h'(\eta)}\sin(\omega)u(\eta)
-\sin(\omega)v(\eta)
\nonumber\\
&=\frac{\sin\omega}{h'(\eta)}q(\eta).
\end{align}
Therefore, since $\sin\omega>0$ and $h'(\eta)<0$,
\cref{e:2xkt} is equivalent to
\begin{equation}
(\forall\tau\in\RP)\quad
\scal{p(\tau)}{q(\eta)}\geq 0.
\end{equation}
Take $\tau\in\RP$ and let us consider the following cases.

(a) $\floor{\tau}=k$:
A direct calculation yields
\begin{align}
&\scal{p(\tau)}{q(\eta)}
\nonumber\\
&\hspace{5mm}
=\scal3{h(\tau)e_0+\cos\brk!{\phi(\tau)}e_{k+1}
+\sin\brk!{\phi(\tau)}e_{k+2}}{
\frac{\pi}{2}e_0
-\frac{\pi}{2}h(\eta)u(\eta)
-h'(\eta)v(\eta)}
\nonumber\\
&\hspace{5mm}
=\frac{\pi}{2}h(\tau)
-\frac{\pi}{2}h(\eta)\brk!{
\cos\brk{\phi(\tau)}\cos\brk{\phi(\eta)}
+\sin\brk{\phi(\tau)}\sin\brk{\phi(\eta)}}
\nonumber\\
&\hspace{5mm}
\phantom{=}
+h'(\eta)\brk!{\cos\brk{\phi(\tau)}\sin\brk{\phi(\eta)}
-\sin\brk{\phi(\tau)}\cos\brk{\phi(\eta)}}
\nonumber\\
&\hspace{5mm}
=\frac{\pi}{2}h(\tau)
-\frac{\pi}{2}h(\eta)\cos\brk!{\phi(\tau)-\phi(\eta)}
-h'(\eta)\sin\brk!{\phi(\tau)-\phi(\eta)}
\nonumber\\
&\hspace{5mm}
=\frac{\pi}{2}h(\tau)
-\frac{\pi}{2}h(\eta)\cos\frac{\pi(\tau-\eta)}{2}
-h'(\eta)\sin\frac{\pi(\tau-\eta)}{2}.
\end{align}
Therefore, since $h>0$ is convex on
$\intv[r]{1}{\pinf}$ (\cref{l:1}\cref{l:1iv})
and since $\abs{\tau-\eta}\leq 1$,
\cref{l:2} ensures that $\scal{p(\tau)}{q(\eta)}\geq 0$.

(b) $\floor{\tau}=k-1$:
Since $h>0$ is convex on $\intv[r]{1}{\pinf}$,
we derive from \cref{l:2} and \cref{l:1}\cref{l:1i} that
\begin{align}
&\scal{p(\tau)}{q(\eta)}
\nonumber\\
&\hspace{5mm}
=\scal3{h(\tau)e_0+\cos\brk!{\phi(\tau)}e_k
+\sin\brk!{\phi(\tau)}e_{k+1}}{
\frac{\pi}{2}e_0
-\frac{\pi}{2}h(\eta)u(\eta)
-h'(\eta)v(\eta)}
\nonumber\\
&\hspace{5mm}
=\frac{\pi}{2}h(\tau)
+\sin\brk!{\phi(\tau)}\brk2{
-\frac{\pi}{2}h(\eta)\cos\brk!{\phi(\eta)}
+h'(\eta)\sin\brk!{\phi(\eta)}}
\nonumber\\
&\hspace{5mm}
=\frac{\pi}{2}h(\tau)
+\sin\brk!{\phi(\tau)}\brk2{
-\frac{\pi}{2}h(\eta)\cos\frac{\pi(k-\eta)}{2}
-h'(\eta)\sin\frac{\pi(k-\eta)}{2}
}
\nonumber\\
&\hspace{5mm}
\geq\frac{\pi}{2}h(\tau)
-\frac{\pi}{2}\sin\brk!{\phi(\tau)}h(k)
\nonumber\\
&\hspace{5mm}
\geq\frac{\pi}{2}h(\tau)-\frac{\pi}{2}h(k)
\nonumber\\
&\hspace{5mm}
\geq 0,
\end{align}
where the last inequality follows from the fact that
$h$ is decreasing on $\RP$ and that $\tau<k<\eta$.

(c) $\floor{\tau}=k+1$:
A direct calculation yields
\begin{align}
&\scal{p(\tau)}{q(\eta)}
\nonumber\\
&\hspace{5mm}
=\scal3{h(\tau)e_0+\cos\brk!{\phi(\tau)}e_{k+2}
+\sin\brk!{\phi(\tau)}e_{k+3}}{
\frac{\pi}{2}e_0
-\frac{\pi}{2}h(\eta)u(\eta)
-h'(\eta)v(\eta)}
\nonumber\\
&\hspace{5mm}
=\frac{\pi}{2}h(\tau)
+\cos\brk!{\phi(\tau)}\brk2{
-\frac{\pi}{2}h(\eta)\sin\brk!{\phi(\eta)}
-h'(\eta)\cos\brk!{\phi(\eta)}}
\nonumber\\
&\hspace{5mm}
=\frac{\pi}{2}h(\tau)
+\cos\brk!{\phi(\tau)}\cos\brk!{\phi(\eta)}h(\eta)\brk2{
-\frac{\pi}{2}\tan\brk!{\phi(\eta)}
+300\eta^2}.
\end{align}
On the other hand, since $\phi(\eta)\in\intv[o]{0}{\pi/2}$
and $k<\eta<k+1-\delta=k+1-1/(300k^2)$,
we derive from \cref{l:0}\cref{l:0v} that
\begin{equation}
\frac{\pi}{2}\tan\phi(\eta)
=\frac{\pi}{2}\cot\brk2{\frac{\pi}{2}-\phi(\eta)}
\leq\frac{\pi}{2}\frac{1}{\frac{\pi}{2}-\phi(\eta)}
=\frac{1}{k+1-\eta}
<300k^2
<300\eta^2.
\end{equation}
Hence $\scal{p(\tau)}{q(\eta)}\geq 0$.

(d) $\abs{\floor{\tau}-k}\geq 2$:
In this case,
\begin{equation}
\scal{e_{\floor{\tau}+1}}{u(\eta)}
=\scal{e_{\floor{\tau}+2}}{u(\eta)}
=\scal{e_{\floor{\tau}+1}}{v(\eta)}
=\scal{e_{\floor{\tau}+2}}{v(\eta)}
=0
\end{equation}
and therefore $\scal{p(\tau)}{q(\eta)}=(\pi/2)h(\tau)>0$.

\cref{p:1ii}:
Since $\eta\geq 1$, $h(\eta)\leq h(1)<1/4$.
Thus, because $\beta\in\intv{0}{1}$ and ${-}h'(\eta)>0$,
we derive from \cref{l:0}\cref{l:0iii} and \cref{e:beeq} that
\begin{equation}
\label{e:cbg4}
\eta-\xi
\leq\sin\frac{\pi(\eta-\xi)}{2}
=\sin\omega
=\frac{-2h'(\eta)(\beta+h(\eta)\cos\omega)}{\pi
\brk{1+h(\eta)^2}}
\leq\frac{-5h'(\eta)}{2\pi}.
\end{equation}
Therefore, using \cref{e:zcfy},
\cref{l:0}\cref{l:0iv}, and the inequality
$\sin\omega\leq\omega$,
we get
\begin{equation}
1-\alpha
=1-\cos\omega+\frac{\pi}{600\eta^2}\sin\omega
\leq\frac{\omega^2}{2}
+\frac{\pi^2(\eta-\xi)}{1200\eta^2}
<{-}\frac{5\pi h'(\eta)(\eta-\xi)}{16}
+\frac{\pi^2(\eta-\xi)}{1200\eta^2}.
\end{equation}

\cref{p:1iii}:
Since $\sin\omega>0$, \cref{e:zcfy} entails that
$\alpha<\cos\omega<1$.
Next, \cref{l:1}\cref{l:1iv} yields
$0>h'(\eta)\geq h'(1)={-}300\exp(-100)$.
Combining this with \cref{p:1ii} and \cref{e:hlxu}, we obtain
\begin{equation}
\alpha
>1+\frac{5\pi h'(\eta)(\eta-\xi)}{16}
-\frac{\pi^2(\eta-\xi)}{1200\eta^2}
\geq 1-\frac{375\pi\exp(-100)}{4}-\frac{\pi^2}{1200}
>\frac{1}{2}.
\end{equation}

\cref{p:1iv}:
First, because $\alpha<1$,
$\gamma-\beta=\beta(1-\alpha)/\alpha+h(\eta)>0$.
To establish the other assertion,
we will employ the inequalities ${-}h'>0$ and $0<h(\eta)<1$,
as well as the identity \cref{e:beeq} repeatedly. We have
\begin{equation}
\label{e:jwf1}
\beta h(\eta)
={-}\frac{\beta h'(\eta)}{300\eta^2}
\leq\frac{-2h'(\eta)(\beta+h(\eta)\cos\omega)}{
300\eta^2\brk{1+h(\eta)^2}}
=\frac{\pi}{300\eta^2}\sin\omega
\leq\frac{\pi}{300\eta^2}\omega
=\frac{\pi^2(\eta-\xi)}{600\eta^2}.
\end{equation}
Next, since $0<{-}h'(\eta)\leq 300\exp(-100)$,
it results from \cref{e:cbg4} that
$\cos\omega>1/2$, which leads to
\begin{align}
h(\eta)^2
&<\frac{4h(\eta)(\beta+h(\eta)\cos\omega)}{1+h(\eta)^2}
\nonumber\\
&=\frac{-4h'(\eta)(\beta+h(\eta)\cos\omega)}{300\eta^2
(1+h(\eta)^2)}
\nonumber\\
&=\frac{\pi}{150\eta^2}\sin\omega
\nonumber\\
&\leq\frac{\pi}{150\eta^2}\omega
\nonumber\\
&=\frac{\pi^2(\eta-\xi)}{300\eta^2}.
\end{align}
At the same time, \cref{l:1}\cref{l:1v} implies that
${-}h'(\eta)\leq 300\exp(-100)/\eta^2$, from which and
\cref{p:1ii} we obtain
\begin{equation}
1-\alpha
\leq\frac{1500\pi(\eta-\xi)\exp(-100)}{16\eta^2}
+\frac{\pi^2(\eta-\xi)}{1200\eta^2}
\leq\frac{\pi^2(\eta-\xi)}{800\eta^2}.
\end{equation}
Altogether, we derive from \cref{p:1iii} that
\begin{equation}
\gamma^2-\beta^2
=\frac{\beta^2(1+\alpha)(1-\alpha)}{\alpha^2}
+\frac{2\beta h(\eta)}{\alpha}
+h(\eta)^2
\leq 8(1-\alpha)
+4\beta h(\eta)
+h(\eta)^2
\leq\frac{\pi^2(\eta-\xi)}{50\eta^2},
\end{equation}
which completes the proof.
\end{proof}

We now return to the Douglas--Rachford iteration \cref{e:1}
and sketch our path to proving \cref{cex:1}, following the
strategy developed in \cite{Hundal04} for the
method of alternating projections.
Starting from $y_0=e_2$, we apply the Douglas--Rachford operator
iteratively, as long as the projection formula of \cref{p:1}
remains applicable for determining $\proj_K(2x_n-y_n)$.
It turns out that, after finitely many iterations, we have
\begin{equation}
y_m=T^my_0=\lambda_1\brk!{\rho_1e_0+u(\xi)}
\end{equation}
for some $m\in\NN\smallsetminus\set{0}$,
$\brk{\rho_1,\lambda_1}\in\intv{0}{1}^2$, and $2-1/150<\xi<2$.
Additionally, $y_m$ is close to $\lambda_1(\rho_1e_0+e_3)$.
Thus, at iteration $m$, \cref{p:1} is no longer sufficient for
determining
\begin{equation}
\proj_K(2x_m-y_m)
=\lambda_1\proj_K(u(\xi)-\rho_1e_0).
\end{equation}
To overcome this difficulty, let us employ the idea found in
\cite{Hundal04}. Set
\begin{equation}
z_0=y_0
\quad\text{and}\quad
z_1=\lambda_1(\rho_1e_0+e_3)
\end{equation}
and consider a ``local'' Douglas--Rachford iteration
$(T^jz_1)_{j\in\NN}$. In a similar fashion, at some iteration
$j\in\NN\smallsetminus\set{0}$, we have
\begin{equation}
T^jz_1=\lambda_2\brk!{\rho_2e_0+u(\eta)}
\end{equation}
for some $(\rho_2,\lambda_2)\in\intv{0}{1}^2$
and some $3-2/(300\times 2^2)<\eta<3$,
and $T^jz_1$ is close to $z_2=\lambda_2(\rho_2e_0+e_4)$.
Repeating this procedure, we construct a sequence
$(z_n)_{n\in\NN}$ of the form
\begin{equation}
(\forall n\in\NN)\quad
z_n=\lambda_n(\rho_ne_0+e_{n+2})
\end{equation}
and a strictly increasing sequence $(l_n)_{n\in\NN}$ in $\NN$
such that $(x_{l_n})_{n\in\NN}$ is sufficiently close in norm
to $(\lambda_ne_{n+2})_{n\in\NN}$ to ensure that
\begin{equation}
\inf_{n\in\NN}\norm{x_{l_n}}>0.
\end{equation}

Let us now make precise the construction of $z_{n+1}$ from
$z_n=\lambda_n(\rho_ne_0+e_{n+2})$.

\begin{proposition}
\label{p:5}
Let $k\in\NN\smallsetminus\set{0}$,
set $\delta=1/(300k^2)$,
let $\rho\in\intv[r]{0}{1}$ be such that
\begin{equation}
\label{e:huyd} 
\rho^2\leq
\frac{\pi^2}{50}\brk2{1-\frac{1}{k}},
\end{equation}
and let $\lambda\in\rzeroun$.
Set $\xi_0=k$, $\beta_0=\rho$, and $\chi_0=\lambda$.
Iterate
\begin{equation}
\label{e:1000}
\begin{array}{l}
\text{for $j=0,1,\ldots$}\\
\left\lfloor
\begin{array}{l}
\text{if $\xi_j\in\intv{k}{k+1-2\delta}$}\\
\left\lfloor
\begin{array}{l}
\text{$\xi_{j+1}$ is the unique point in
$\intv[o]{\xi_j}{k+1-\delta}$ such that}\\
\quad\;\;\displaystyle
\beta_j={-}h(\xi_{j+1})\cos\omega_j
-\frac{\pi}{2}\frac{1+h(\xi_{j+1})^2}{h'(\xi_{j+1})}\sin\omega_j,
\quad\text{where}\,\,\omega_j=\frac{\pi(\xi_{j+1}-\xi_j)}{2}
\smallskip\\
\displaystyle
\alpha_j=\cos\omega_j
+\frac{\pi}{2}\frac{h(\xi_{j+1})}{h'(\xi_{j+1})}
\sin\omega_j\smallskip\\
\beta_{j+1}=\beta_j/\alpha_j+h(\xi_{j+1})\smallskip\\
\chi_{j+1}=\alpha_j\chi_j
\end{array}
\right.\\
\text{else}\\
\left\lfloor
\text{terminate}.
\right.\\[2mm]
\end{array}
\right.
\end{array}
\end{equation}
Set
\begin{equation}
m=1+\sup\menge{j\in\NN}{\text{\cref{e:1000} does not terminate
at iteration $j$}}.
\end{equation}
Then $m\geq 1$ and the following hold:
\begin{enumerate}
\item
\label{p:5i}
Let $j\in\NN$ be such that $j+1\leq m$.
Then the following hold:
\begin{enumerate}
\item
$\alpha_j\in\intv[o]{1/2}{1}$.
\item
$\beta_j<\beta_{j+1}$.
\item
\label{p:5ic}
$\beta_{j+1}^2-\beta_0^2\leq(\pi^2/50)(1/k-1/(k+1))$.
\item
\label{p:5id}
$\beta_{j+1}^2\leq(\pi^2/50)(1-1/(k+1))$.
\end{enumerate}
\item
\label{p:5ii}
$m<\pinf$.
\item
\label{p:5iii}
Set
\begin{equation}
z=\lambda(\rho e_0+e_{k+1}),
\quad
\rho_+=\beta_m,
\quad
\lambda_+=\chi_m,
\quad\text{and}\quad
z_+=\lambda_+(\rho_+e_0+e_{k+2}).
\end{equation}
Then the following hold:
\begin{enumerate}
\item
\label{p:5iiia}
$\rho_+^2\leq(\pi^2/50)(1-1/(k+1))$.
\item
\label{p:5iiib}
$\rho_+^2-\rho^2\leq(\pi^2/50)(1/k-1/(k+1))$.
\item
\label{p:5iiic}
${-}\ln(\lambda_+/\lambda)\leq
(5\pi/8)(h(k)-h(k+1))+(\pi^2/600)(1/k-1/(k+1))$.
\item
\label{p:5iiid}
$T^mz=\lambda_+(\rho_+e_0+u(\xi_m))$.
\item
\label{p:5iiie}
$\norm{z_+-T^mz}\leq\pi/(300k^2)$.
\end{enumerate}
\end{enumerate}
\end{proposition}
\begin{proof}
In the light of \cref{p:1},
\begin{equation}
\label{e:5jru}
\text{\cref{e:1000} does not terminate at iteration $j=0$}.
\end{equation}
Therefore $\xi_1$, $\alpha_0$, and $\beta_1$ are well defined, with
$\xi_1>\xi_0$. Furthermore, $m\geq 1$.

\cref{p:5i}:
We proceed by strong induction.
If $j=0$, \cref{p:1}\cref{p:1iii} (applied to $\xi_0=k$ and
$\beta_0\in\lzeroun$) states that
$\alpha_0\in\intv[o]{1/2}{1}$, while it results from
\cref{p:1}\cref{p:1iv} that $\beta_1>\beta_0$ and
\begin{equation}
\label{e:tkav}
\beta_1^2-\beta_0^2
\leq\frac{\pi^2(\xi_1-\xi_0)}{50\xi_1^2}
\leq\frac{\pi^2}{50}\brk2{\frac{1}{\xi_0}-\frac{1}{\xi_1}}
\leq\frac{\pi^2}{50}\brk2{\frac{1}{k}-\frac{1}{k+1}},
\end{equation}
from which and \cref{e:huyd} we obtain
\begin{equation}
\beta_1^2\leq\frac{\pi^2}{50}\brk2{1-\frac{1}{k+1}}.
\end{equation}
Now assume that $m\geq 2$ and that, for some
$j\in\NN$ such that $j+2\leq m$, the assertions hold for every
$i\in\set{0,\ldots,j}$.
Then
\begin{equation}
(\forall i\in\set{0,\ldots,j})\quad
\begin{cases}
0\leq\beta_i<\beta_{i+1}<1\\
\text{$\xi_{i+1}$ is well defined}\\
k\leq\xi_i<\xi_{i+1}\leq k+1-2\delta.
\end{cases}
\end{equation}
Therefore, by virtue of \cref{p:1}
(applied to $\xi_{j+1}$ and $\beta_{j+1}$),
we deduce that
$\xi_{j+2}$ is well defined and
$\xi_{j+1}<\xi_{j+2}<k+1-\delta$.
Furthermore, \cref{p:1}\cref{p:1iii} asserts that
$\alpha_{j+1}\in\intv[o]{1/2}{1}$, while \cref{p:1}\cref{p:1iv}
asserts that $\beta_{j+2}>\beta_{j+1}$.
In turn, for every $i\in\set{0,\ldots,j+1}$,
it results from \cref{p:1}\cref{p:1iv}
(applied to $\xi_i$ and $\beta_i$) that
\begin{equation}
\beta_{i+1}^2-\beta_i^2
\leq\frac{\pi^2(\xi_{i+1}-\xi_i)}{50\xi_{i+1}^2}
\leq\frac{\pi^2}{50}\brk2{\frac{1}{\xi_i}-\frac{1}{\xi_{i+1}}}.
\end{equation}
Thus
\begin{equation}
\label{e:tppv}
\beta_{j+2}^2-\beta_0^2
=\sum_{i=0}^{j+1}(\beta_{i+1}^2-\beta_i^2)
\leq\frac{\pi^2}{50}\sum_{i=0}^{j+1}
\brk2{\frac{1}{\xi_i}-\frac{1}{\xi_{i+1}}}
=\frac{\pi^2}{50}\brk2{\frac{1}{\xi_0}-\frac{1}{\xi_{j+2}}}
\leq\frac{\pi^2}{50}\brk2{\frac{1}{k}-\frac{1}{k+1}},
\end{equation}
from which and \cref{e:huyd} we obtain
$\beta_{j+2}^2\leq(\pi^2/50)(1-1/(k+1))$.

\cref{p:5ii}:
Assume to the contrary that $m=\pinf$.
Then \cref{e:1000} generates infinite sequences
$(\xi_j)_{j\in\NN}$ and $(\beta_j)_{j\in\NN}$, and the following
hold:
\begin{itemize}
\item
$(\forall j\in\NN)$ $k\leq\xi_j<\xi_{j+1}\leq k+1-2\delta$.
\item
$(\beta_j)_{j\in\NN}$ is strictly increasing and
$(\forall j\in\NN)$ $\beta_{j+1}^2\leq(\pi^2/50)(1-1/(k+1))$.
\end{itemize}
It follows that $(\xi_j)_{j\in\NN}$ converges to some
$\xi_{\infty}\in\intv[l]{k}{k+1-2\delta}$
and $(\beta_j)_{j\in\NN}$ converges to some
$\beta_{\infty}\in\zeroun$.
Therefore, since \cref{e:1000} entails that
\begin{equation}
(\forall j\in\NN)\quad
\beta_j={-}h(\xi_{j+1})\cos\frac{\pi(\xi_{j+1}-\xi_j)}{2}
-\frac{\pi}{2}\frac{1+h(\xi_{j+1})^2}{h'(\xi_{j+1})}
\sin\frac{\pi(\xi_{j+1}-\xi_j)}{2},
\end{equation}
letting $j\to\pinf$ yields $0<\beta_{\infty}={-}h(\xi_{\infty})<0$,
which is absurd.

\cref{p:5iiia}:
Apply \cref{p:5id} to $j=m-1$.

\cref{p:5iiib}:
Apply \cref{p:5ic} to $j=m-1$.

\cref{p:5iiic}:
Note that $(\forall\alpha\in\intv[o]{1/2}{1})$
${-}\ln\alpha\leq 2(1-\alpha)$.
Thus, we derive from \cref{e:1000},
\cref{p:5i}, and \cref{p:1}\cref{p:1ii} that
\begin{equation}
{-}\ln\frac{\lambda_+}{\lambda}
={-}\sum_{j=0}^{m-1}\ln\alpha_j
\leq 2\sum_{j=0}^{m-1}(1-\alpha_j)
\leq\frac{5\pi}{8}
\sum_{j=0}^{m-1}{-}h'(\xi_{j+1})(\xi_{j+1}-\xi_j)
+\frac{\pi^2}{600}\sum_{j=0}^{m-1}
\frac{\xi_{j+1}-\xi_j}{\xi_{j+1}^2}.
\end{equation}
On the other hand, applying
\cref{l:3} to the continuous decreasing functions
$\intv{k}{k+1}\ni t\mapsto{-}h'(t)$
and $\intv{k}{k+1}\ni t\mapsto 1/t^2$, we obtain
\begin{equation}
\sum_{j=0}^{m-1}{-}h'(\xi_{j+1})(\xi_{j+1}-\xi_j)
\leq\int_k^{k+1}{-}h'(t)dt
=h(k)-h(k+1)
\end{equation}
and
\begin{equation}
\sum_{j=0}^{m-1}\frac{\xi_{j+1}-\xi_j}{\xi_{j+1}^2}
\leq\int_k^{k+1}\frac{dt}{t^2}
=\frac{1}{k}-\frac{1}{k+1},
\end{equation}
respectively.

\cref{p:5iiid}:
We show that
\begin{equation}
\label{e:Tind}
\brk{\forall j\in\set{1,\ldots,m}}\quad
T^jz=\chi_j\brk!{\beta_je_0+u(\xi_j)}
\end{equation}
by induction on $j\in\set{1,\ldots,m}$.
In fact, since $V=\set{e_0}^{\perp}$,
$\proj_Vz=\lambda e_{k+1}=\chi_0e_{k+1}$.
Thus, in view of \cref{e:1000}, we derive from
\cite[Proposition~29.29]{Livre1} and \cref{p:1}\cref{p:1i} that
\begin{equation}
\proj_K(2\proj_Vz-z)
=\proj_K(\chi_0e_{k+1}-\chi_0\beta_0e_0)
=\chi_0\proj_K\brk!{u(\xi_0)-\beta_0e_0}
=\chi_1p(\xi_1)
\end{equation}
and, in turn, from \cref{e:drop} that
\begin{equation}
Tz
=\chi_0\beta_0e_0+\chi_1p(\xi_1)
=\chi_1\brk2{\frac{\beta_0}{\alpha_0}e_0+h(\xi_1)e_0+u(\xi_1)}
=\chi_1\brk!{\beta_1e_0+u(\xi_1)}.
\end{equation}
Now assume that $m\geq 2$ and that, for some
$j\in\set{1,\ldots,m-1}$, \cref{e:Tind} holds, and
set $y=T^jz=\chi_j\brk{\beta_je_0+u(\xi_j)}$.
Then, because $u(\xi_j)\in\set{e_0}^{\perp}=V$, we get
$\proj_Vy=\chi_ju(\xi_j)$.
Thus, invoking \cref{e:1000},
we deduce from
\cite[Proposition~29.29]{Livre1} and \cref{p:1}\cref{p:1i} that
\begin{equation}
\proj_K(2\proj_Vy-y)
=\chi_j\proj_K\brk!{u(\xi_j)-\beta_je_0}
=\chi_{j+1}p(\xi_{j+1}),
\end{equation}
from which and \cref{e:drop} it follows that
\begin{align}
T^{j+1}z
&=Ty
\nonumber\\
&=\chi_j\beta_je_0+\chi_{j+1}p(\xi_{j+1})
\nonumber\\
&=\chi_{j+1}\brk2{\frac{\beta_j}{\alpha_j}e_0
+h(\xi_{j+1})e_0+u(\xi_{j+1})}
\nonumber\\
&=\chi_{j+1}\brk!{\beta_{j+1}e_0+u(\xi_{j+1})}.
\end{align}
Hence \cref{e:Tind} holds.

\cref{p:5iiie}:
Because $\set{\alpha_j}_{0\leq j\leq m-1}\subset\zeroun$,
it results from \cref{e:1000} that
$\lambda_+
=\chi_m
=\lambda\prod_{j=0}^{m-1}\alpha_j
\in\zeroun$.
On the other hand, since \cref{e:1000} terminates at iteration
$m$, we infer that $\xi_m>k+1-2\delta$. Hence,
since $\floor{\xi_m}=k$ by construction,
it follows from \cref{p:5iiid} that
\begin{align}
\norm{z_+-T^mz}
&\leq\norm{e_{k+2}-u(\xi_m)}
\nonumber\\
&=\sqrt{\cos^2\phi(\xi_m)+\brk!{1-\sin\phi(\xi_m)}^2}
\nonumber\\
&=\sqrt{2-2\sin\phi(\xi_m)}
\nonumber\\
&=2\abs3{\sin\brk2{\frac{\pi}{4}-\frac{\phi(\xi_m)}{2}}}
\nonumber\\
&=2\sin\frac{\pi(k+1-\xi_m)}{4}
\nonumber\\
&\leq\frac{\pi(k+1-\xi_m)}{2}
\nonumber\\
&<\pi\delta
\nonumber\\
&=\frac{\pi}{300k^2},
\end{align}
which completes the proof.
\end{proof}

We are now in a position to prove \cref{cex:1}.

\begin{proof}[\textcolor{structure}{\textit{Proof
of \cref{cex:1}}}]
\cref{t:1} states that there exists $(x,y)\in\HH\times\HH$ such
that
\begin{equation}
\label{e:xyli}
x=\proj_Vy,
\quad
x\in V\cap K,
\quad
x=\proj_K(2x-y),
\quad
x_n\weakly x,
\quad\text{and}\quad
y_n\weakly y.
\end{equation}
However, as shown in \cite{Hundal04}, $V\cap K=\set{0}$. Therefore
\begin{equation}
x_n\weakly 0.
\end{equation}
Next, let us show that $y\in\RPP\cdot e_0$.
In fact, it results from \cref{e:xyli} that
$\proj_Vy=0$ and $\proj_K(-y)=0$.
Thus, by virtue of Moreau's conical decomposition
\cite[Theorem~6.30(i)]{Livre1},
$y\in V^{\perp}=\RR e_0$
and $(\forall w\in K)$
$\scal{w}{-y}\leq 0$. This implies that there exists
$\zeta\in\RR$ such that
\begin{equation}
y=\zeta e_0
\quad\text{and}\quad
(\forall w\in K)\,\,
\scal{w}{-\zeta e_0}\leq 0.
\end{equation}
Since $e_0+e_1=p(0)\in K$, it follows that $\zeta\geq 0$.
To show that $\zeta>0$, define
\begin{equation}
(\forall n\in\NN)\quad
\zeta_n=\scal{y_n}{e_0}
\end{equation}
and observe that, by construction of $K$,
\begin{equation}
(\forall w\in K)\quad
\scal{w}{e_0}\geq 0.
\end{equation}
Because $\set{x_n}_{n\in\NN}\subset\set{e_0}^{\perp}$,
we derive from \cref{e:1} that
\begin{equation}
(\forall n\in\NN)\quad
\zeta_{n+1}
=\zeta_n+\scal{\proj_K(2x_n-y_n)}{e_0}
\geq\zeta_n.
\end{equation}
At the same time, because $y_n\weakly y$,
$\zeta_n=\scal{y_n}{e_0}\to\scal{y}{e_0}=\zeta$.
Hence, it suffices to show that $\zeta_1>0$.
Indeed, since $\proj_K(2x_0-y_0)=\proj_Ke_2=\proj_Ku(1)$,
\cref{p:1} (applied to $k=1$ and $\beta=0$) entails that there
exist $\eta\in\intv[o]{1}{2}$ and $\alpha\in\zeroun$ such that
$\proj_K(2x_0-y_0)=\alpha p(\eta)$. Thus 
$y_1=\alpha p(\eta)$, which leads to
$\zeta_1=\alpha h(\eta)>0$, as desired.
Next, we show that $x_n\not\to 0$ and that $y_n\not\to y$.
To do so, set
\begin{equation}
\rho_0=0,
\quad
\lambda_0=1,
\quad\text{and}\quad
z_0=\lambda_0(\rho_0e_0+e_2).
\end{equation}
Using induction, we construct a sequence
$(z_n)_{n\in\NN}$ in $\HH$ of the form
\begin{equation}
\label{e:znfo}
(\forall n\in\NN)\quad
z_n=\lambda_n(\rho_ne_0+e_{n+2}),
\end{equation}
where $\set{\rho_n}_{n\in\NN}\subset\lzeroun$
and $\set{\lambda_n}_{n\in\NN}\subset\rzeroun$,
together with a sequence $(m_n)_{n\in\NN}$ in
$\NN\smallsetminus\set{0}$ such that
\begin{equation}
\label{e:znse}
(\forall n\in\NN)\quad
\begin{cases}
\displaystyle
\rho_n^2\leq\frac{\pi^2}{50}\brk2{1-\frac{1}{n+1}}
\medskip\\
\displaystyle
\rho_{n+1}^2-\rho_n^2\leq\frac{\pi^2}{50}
\brk2{\frac{1}{n+1}-\frac{1}{n+2}}
\medskip\\
\displaystyle
{-}\ln\frac{\lambda_{n+1}}{\lambda_n}
\leq\frac{5\pi}{8}\brk!{h(n+1)-h(n+2)}
+\frac{\pi^2}{600}\brk2{\frac{1}{n+1}-\frac{1}{n+2}}
\medskip\\
\displaystyle
\norm{z_{n+1}-T^{m_n}z_n}\leq\frac{\pi}{300(n+1)^2}.
\end{cases}
\end{equation}
Assume that, for some $n\in\NN$,
$z_n=\lambda_n(\rho_ne_0+e_{n+2})$ has already been constructed,
where $\lambda_n\in\rzeroun$ and $\rho_n\in\lzeroun$ satisfies
$\rho_n^2\leq(\pi^2/50)\brk{1-1/(n+1)}$.
(This is true when $n=0$.)
We then obtain $m_n\in\NN\smallsetminus\set{0}$,
$\rho_{n+1}\in\lzeroun$, and $\lambda_{n+1}\in\rzeroun$
by applying \cref{p:5} to
$(k,\rho,\lambda)=(n+1,\rho_n,\lambda_n)$.
We derive from \cref{e:znse} that
\begin{align}
(\forall n\in\NN)\quad
\ln\lambda_{n+1}
&=\sum_{i=0}^n\ln\frac{\lambda_{i+1}}{\lambda_i}
\nonumber\\
&\geq\frac{5\pi}{8}\brk!{h(n+2)-h(1)}
+\frac{\pi^2}{600}\brk2{\frac{1}{n+2}-1}
\nonumber\\
&>{-}\frac{5\pi\exp(-100)}{8}-\frac{\pi^2}{600}
\nonumber\\
&>3\ln\pi-\ln(1800),
\end{align}
which implies that
\begin{equation}
\label{e:infl}
\inf_{n\in\NN}\lambda_n
\geq\exp\brk2{{-}\frac{5\pi\exp(-100)}{8}-\frac{\pi^2}{600}}
>\frac{\pi^3}{1800}.
\end{equation}
Next, define a strictly increasing sequence
$(l_n)_{n\in\NN}$ inductively via
\begin{equation}
\label{e:lnde}
l_0=0
\quad\text{and}\quad
(\forall n\in\NN)\,\,
l_{n+1}=l_n+m_n.
\end{equation}
We claim that
\begin{equation}
\label{e:ylnz}
(\forall n\in\NN)\quad
\norm{z_n-y_{l_n}}\leq\frac{\pi}{300}\sum_{i=0}^{n-1}
\frac{1}{(i+1)^2},
\end{equation}
with the convention that $\sum_{i=0}^{-1}(i+1)^{-2}=0$.
When $n=0$, the claim is trivial since $z_0=e_2=y_0=y_{l_0}$.
Assume that, for some $n\in\NN$,
$\norm{z_n-y_{l_n}}\leq(\pi/300)\sum_{i=0}^{n-1}(i+1)^{-2}$.
Invoking \cref{e:lnde}, \cref{e:znse}, and the nonexpansiveness of
$T$ \cite[Theorem~26.11]{Livre1}, we obtain
\begin{align}
\norm{z_{n+1}-y_{l_{n+1}}}
&=\norm{z_{n+1}-T^{m_n}z_n+T^{m_n}z_n-T^{m_n}y_{l_n}}
\nonumber\\
&\leq\norm{z_{n+1}-T^{m_n}z_n}+\norm{T^{m_n}z_n-T^{m_n}y_{l_n}}
\nonumber\\
&\leq\norm{z_{n+1}-T^{m_n}z_n}+\norm{z_n-y_{l_n}}
\nonumber\\
&\leq\frac{\pi}{300(n+1)^2}
+\frac{\pi}{300}\sum_{i=0}^{n-1}\frac{1}{(i+1)^2}
\nonumber\\
&=\frac{\pi}{300}\sum_{i=0}^n\frac{1}{(i+1)^2}.
\end{align}
In turn, appealing to \cref{e:1}, \cref{e:znfo}, and the
nonexpansiveness of $\proj_V$, we obtain
\begin{align}
(\forall n\in\NN)\quad
\lambda_n-\norm{x_{l_n}}
&=\norm{\lambda_ne_{n+2}}-\norm{x_{l_n}}
\nonumber\\
&\leq\norm{\lambda_ne_{n+2}-x_{l_n}}
\nonumber\\
&=\norm{\proj_Vz_n-\proj_Vy_{l_n}}
\nonumber\\
&\leq\norm{z_n-y_{l_n}}
\nonumber\\
&\leq\frac{\pi}{300}\sum_{i=0}^{n-1}\frac{1}{(i+1)^2}
\nonumber\\
&\leq\frac{\pi^3}{1800},
\end{align}
where the last inequality follows from the fact that
$\sum_{i=1}^{\pinf}1/i^2=\pi^2/6$.
Hence, it results from \cref{e:infl}
and the nonexpansiveness of $\proj_V$ that
\begin{equation}
\inf_{n\in\NN}\norm{x_{l_n}}
\geq\inf_{n\in\NN}\lambda_n-\frac{\pi^3}{1800}
>0
\end{equation}
and, therefore, that $x_n\not\to 0$.
Finally, if $(y_n)_{n\in\NN}$ converged strongly to $y$,
the continuity of $\proj_V$ and \cref{e:xyli} would yield
$x_n=\proj_Vy_n\to\proj_Vy=x=0$, which leads to a contradiction.
\end{proof}

\paragraph{\textcolor{structure}{Acknowledgment}}{%
The work of M.~N.~B\`{u}i was partially supported by
the National Natural Science Foundation of China
under grant number 12671353.
The work of H.~Luo was partially supported by
the National Natural Science Foundation of China
under grant number 12671353
and by the Chongqing Natural Science Foundation under grant number
KJZD-K202500507.
The work of X.~Wang was partially supported by an
NSERC Discovery Grant.
}

\bibliographystyle{jnsao}
\bibliography{bibli}

\end{document}